%% file: Main.tex
\documentclass[a4paper, reqno]{amsart}

\usepackage{Settings}

\title{Singularity Categories of Simple Singularities \\ in Positive Characteristic}

\author{Yuta Takashima}
\address{(Yuta TAKASHIMA) Department of Mathematical Sciences, Graduate School of Science,
Tokyo Metropolitan University, 1-1 Minamiosawa, Hachioji, Tokyo 192-0393, Japan.}
\email{yuta.takashima.m@gmail.com}

\subjclass[2020]{
    Primary~14F08; 
    Secondary~14B05, 
    13D03. 
}

\keywords{
    singularity category,
    Hochschild cohomology,
    Tyurina algebra,
    simple singularity,
    rational double point,
    standard category.
}
\begin{document}

\begin{abstract}
    We study the singularity categories of simple singularities of the same dimension
    over an algebraically closed field of positive characteristic, and show that, as in characteristic zero,
    these categories are not equivalent as triangulated categories unless the underlying singularities are
    analytically isomorphic.
    In contrast to the characteristic zero case, simple singularities in positive characteristic
    cannot be distinguished solely from the Auslander--Reiten quivers of their singularity categories.
    To address this, we extend to positive characteristics a theorem by Hua and Keller,
    which asserts that the 0th Hochschild cohomology of the dg singularity category of
    an isolated hypersurface singularity in characteristic zero is isomorphic to the Tyurina algebra
    of the defining polynomial.
    Furthermore as an application, we determine the condition for the singularity category of
    a rational double point (i.e., a simple singularity of dimension two) to be standard.
    We prove that such a category is standard if and only if the defining polynomial is weighted homogeneous.
\end{abstract}

\maketitle

\setcounter{tocdepth}{2}
\tableofcontents

\input{materials/Intro}
\input{materials/Prelim}
\input{materials/MainResult}
\input{materials/App}

\bibliographystyle{alphaurl}
\bibliography{References}
\end{document}

%% file: materials/Intro.tex
\section{Introduction}
Let $A$ and $A'$ be complete Noetherian local algebras over an algebraically closed field $k$.
It is obvious that if $A \cong A'$ then their singularity categories
$\DD_{sg}(A)$ and $\DD_{sg}(A')$ are $k$-linear triangulated equivalent.
Conversely, does a $k$-linear triangulated equivalence $\DD_{sg}(A) \simeq \DD_{sg}(A')$
induce an isomorphism $A \cong A'$?
The answer is no in general because of Kn\"{o}rrer's periodicity (\cite{MR877010} and \cite{MR0977477}).
In this paper, focusing on simple singularities of the same dimension,
we study whether each singularity category determines its underlying singularity.
Representation-theoretically, a \emph{simple singularity} is characterized as an isolated hypersurface singularity
$k\exdbra{x_0, x_1, \ldots, x_d} / \exgen{f}$ with $f \in k[x_0, x_1, \ldots, x_d]$
whose singularity category has only finitely many indecomposable objects up to isomorphism (\cite{MR1033443}). Set
\begin{align*}
    &\mathsf{Sing} \coloneqq \exset{\text{
        simple singularities over $k$ of dimension $d$
    }} / \mathord{\cong}, \\
    &\mathsf{Cat} \coloneqq \inset{\DD_{sg}(R)}{R \in \mathsf{Sing}} / \text{($k$-linear triangulated equivalence)}, \\
    &\mathsf{AR} \coloneqq \inset{(\varGamma, \tau)}{\text{
        $(\varGamma, \tau)$ is the Auslander--Reiten quiver of some $\D \in \mathsf{Cat}$
    }} / \mathord{\cong}.
\end{align*}
Then in characteristic $0$, since the composition
\begin{equation}
    \label{eq:Intro}
    \begin{array}{ccccc}
        \mathsf{Sing} & \twoheadrightarrow & \mathsf{Cat} & \twoheadrightarrow & \mathsf{AR} \\
        R & \mapsto & \DD_{sg}(R) & & \\
        & & \D & \mapsto & \text{(the Auslander--Reiten quiver of $\D$)}
    \end{array}
\end{equation}
is an injection (\cite{MR825715} and \cite{MR0887498}),
the mapping $\mathsf{Sing} \twoheadrightarrow \mathsf{Cat}$ in \cref{eq:Intro} is a bijection.
This means that the singularity categories reflect the differences in the analytic isomorphism classes
of simple singularities.
What happens in positive characteristic? The following is our first main result.

\begin{thm}[{\cref{thm:LongWork}}]
    \label{thm:FirstMain}
    As in characteristic $0$, the mapping $\mathsf{Sing} \twoheadrightarrow \mathsf{Cat}$ in \cref{eq:Intro}
    is a bijection in positive characteristic\textup{:}
    for simple singularities $R$ and $R'$ over $k$ of dimension $d$, they are isomorphic
    if and only if their singularity categories $\DD_{sg}(R)$ and $\DD_{sg}(R')$
    are $k$-linear triangulated equivalent.
\end{thm}

In contrast to the characteristic $0$ case, the composition \cref{eq:Intro}
fails to be an injection in positive characteristic in general (\cite{MR818299} and \cite{MR0887498}),
which implies that simple singularities in positive characteristic cannot be distinguished solely
from the Auslander--Reiten quivers of their singularity categories.
In order to prove \cref{thm:FirstMain}, we extend the next theorem
to the case of positive characteristic (see \cref{rem:HK}).

\begin{thm}[{\cite[Theorem~5.9]{MR4817470}}]
    Assume that the characteristic of $k$ is zero and \linebreak $\Spec \expar{k[x_0, x_1, \ldots, x_d] / \exgen{f}}$
    has only one singular point at the origin. Then
    the $0$th Hochschild cohomology of the dg singularity category
    $\DD_{sg}^{dg}(k\exdbra{x_0, x_1, \ldots, x_d} / \exgen{f})$ is isomorphic to
    the Tyurina algebra of the polynomial $f$ as $k\exdbra{x_0, x_1, \ldots, x_d} / \exgen{f}$-algebras\textup{:}
    \[
        \HH^0(\DD_{sg}^{dg}(k\exdbra{x_0, x_1, \ldots, x_d} / \exgen{f})) \cong
        k\exdbra{x_0, x_1, \ldots, x_d} \bigg/ \exgen{f, \frac{\partial f}{\partial x_0}, \frac{\partial f}{\partial x_1}, \ldots, \frac{\partial f}{\partial x_d}}.
    \]
\end{thm}

Roughly speaking, a category that can be recovered from its Auslander--Reiten quiver is said to be
\emph{standard} (see \cref{dfn:Std}).
As an application of \cref{thm:FirstMain}, we determine the condition for
the singularity category of a rational double point (e.g., a simple singularity of dimension $2$) to be standard,
which is our second main result.

\begin{thm}[{\cref{thm:RDPStd}}]
    \label{thm:SecondMain}
    Let $R \coloneqq k\exdbra{x_0, x_1, x_2} / \exgen{f}$ be a rational double point, and
    $T_n^r$ \textup{(}$T \in \exset{A, D, E}$, $n \ge 1$ and $r \ge 0$\textup{)} the type of $R$ \textup{(}see \cref{tab:Eq2}\textup{)}.
    Then the singularity category $\DD_{sg} (R)$ is standard if and only if $r = 0$.
\end{thm}

In particular, it follows from \cref{thm:SecondMain} that all singularity categories of rational double points
are standard in characteristic 0, which has been known as a folklore result (e.g., \cite[Section~5]{MR0915168}).
At the same time, \cref{thm:SecondMain} shows that this property fails in characteristics $2$, $3$ or $5$.
Consequently, the assertion in \cite[Proposition~3.3]{MR1388043} that they are standard in arbitrary characteristic
turns out to be false.

\begin{cvn}
    \begin{items}
        \item $k$ denotes an algebraically closed field.
        \item Any functor between $k$-linear categories is assumed to be $k$-linear.
        \item Let $A$ be a ring.
        $\mod A$ stands for the category of finitely generated right $A$-modules,
        and $\proj A$ is a full subcategory of $\mod A$ consisting of the projective right $A$-modules.
        \item Let $S$ be a commutative ring.
        $K_{\bullet}(S;f_1, f_2, \ldots, f_n)$ denotes the Koszul complex of a sequence $f_1, f_2, \ldots, f_n \in S$.
        \item Let $\varGamma$ be a quiver.
        We think of paths as \emph{going from right to left}:
        a path $a_n a_{n-1} \cdots a_1$ corresponds to the concatenation
        \[
            \begin{tikzcd}[ampersand replacement=\&]
                \bullet \&
                    \bullet \ar{l}[']{a_n} \&
                        \ar{l}[']{a_{n-1}} \cdots \&
                            \bullet\rlap{.} \ar{l}[']{a_1}
            \end{tikzcd}
        \]
        $k(\varGamma)$ (resp. $k[\varGamma]$) denotes the path category
        (resp. the path algebra) of $\varGamma$. Here,
        \begin{itemize}
            \item the objects in $k(\varGamma)$ consist of the vertices in $\varGamma$;
            \item $\Hom_{k(\varGamma)} (M, N)$ is a $k$-vector space
            whose basis consists of the paths from an object $M$ to another $N$,
            and composition of morphisms is induced by concatenation of paths.
        \end{itemize}
        \item Let $(\varGamma, \tau)$ be a locally finite translation quiver.
        $k(\varGamma, \tau)$ (resp. $k[\varGamma, \tau]$) denotes the mesh category
        (resp. the mesh algebra) of $(\varGamma, \tau)$, which is the quotient of
        the path category (resp. the path algebra) of $\varGamma$ by the mesh ideal.
    \end{items}
\end{cvn}

\begin{ack}
    The author would like to express his gratitude to Professor Hokuto Uehara for suggesting the references
    \cite{MR4404125} and \cite{MR4817470}, and for sharing his email correspondence
    with Professors Zheng Hua and Bernhard Keller, whose comments are presented in \cref{rem:HK}.
\end{ack}

%% file: materials/Prelim.tex
\section{Simple Singularities}
In this section, we collect some preliminaries on simple singularities and their (dg) singularity categories.
Let $S_d \coloneqq k \exdbra{x_0, x_1, \ldots, x_d}$ with $d \geq 0$ be a formal power series ring,
and $f \in k[x_0, x_1, \ldots, x_d]$ a polynomial
such that $\Spec (S_d / \exgen{f})$ has only one singular point at the origin.

\begin{dfn}
    The \emph{Tyurina algebra} of the polynomial $f$ is defined by
    \[ T_f \coloneqq S_d \bigg/ \exgen{f, \frac{\partial f}{\partial x_0}, \frac{\partial f}{\partial x_1}, \ldots, \frac{\partial f}{\partial x_d}}. \]
    The rank $\tau(f)$ of $T_f$ over $k$ is referred to as the \emph{Tyurina number} of $f$.
\end{dfn}

\begin{rem}
    The Tyurina algebra $T_f$ is an invariant of the analytic class:
    if two isolated hypersurface singularities $S_d / \exgen{f}$ and $S_d / \exgen{g}$ are isomorphic,
    then $T_f \cong T_g$.
\end{rem}

\begin{dfn}
    The isolated hypersurface singularity $S_d / \exgen{f}$ is a \emph{simple singularity}
    if the number of indecomposable objects in the category $\MCM(S_d / \exgen{f})$
    of maximal Cohen--Macaulay $S_d / \exgen{f}$-modules is finite up to isomorphism.
\end{dfn}

\begin{eg}
    \label{eg:ZeroSimple}
    The isolated hypersurface singularities of dimension $0$ are exactly of the form
    \begin{align*}
        &k\exdbra{x_0} / \exgen{x_0^{n+1}}
            & &\text{for $n \geq 1$.}
    \end{align*}
    Since the indecomposable maximal Cohen--Macaulay $k\exdbra{x_0} / \exgen{x_0^{n+1}}$-modules consist of
    $k\exdbra{x_0} / \exgen{x_0}, \linebreak k\exdbra{x_0} / \exgen{x_0^2}, \ldots, k\exdbra{x_0} / \exgen{x_0^{n+1}}$
    for $n \geq 1$ (e.g., \cite[Theorem~3.3]{MR2919145}), any isolated hypersurface singularity of dimension $0$ is a simple singularity.
\end{eg}

\begin{thm}[{\cite{MR1033443}}]
    \label{thm:Classification}
    \begin{items}
        \item \label{thm:Classification:1} Let $x_0, x_1$ and $x_2$ also be denoted by $x, y$ and $z$ respectively.
        Assume that $S_d / \exgen{f}$ is a simple singularity of dimension $1$ \textup{(}resp. $2$\textup{)}.
        Then after a suitable change of variables, $f$ is one of the polynomials
        in \cref{tab:Eq1} \textup{(}resp. \cref{tab:Eq2}\textup{)}.
        \input{tables/Eq1}
        \input{tables/Eq2}
        \item \label{thm:Classification:2} Assume $d \geq 1$. Then the mapping
        \[
            \begin{array}{ccc}
                \exset{
                    \text{simple singularities of dimension $d$}
                } / \mathord{\cong}
                    & \to
                        & \exset{
                            \text{simple singularities of dimension $d + 2$}
                        } / \mathord{\cong} \\
                S_d / \exgen{g}
                    & \mapsto
                        & S_{d + 2} / \exgen{g + x_{d + 1} x_{d + 2}}
            \end{array}
        \]
        is a well-defined bijection.
    \end{items}
\end{thm}

\begin{rem}
    \begin{items}
        \item Some of the polynomials in \cref{tab:Eq2} differ from those in \cite{MR1033443}.
        Simple changes of variables show that both lists are essentially the same.
        For type $E_6^1$ in characteristic $3$, for instance, the polynomial
        $z^2 + x^3 + y^2z + xyz$ in \cref{tab:Eq2} transforms into $z^2 + x^3 + y^4 + x^2y^2$ in \cite{MR1033443}
        as follows:
        \begin{align*}
            &z^2 + x^3 + y^2z + xyz \\
            &= z^2 + (-x^2 + y_1^2)z + x^3 & &\text{$(y_1 \coloneqq -x + y)$} \\
            &= z_1^2 + x^3(1 - x) - y_1^4 - x^2y_1^2 & &\text{$(z_1 \coloneqq z + x^2 - y_1^2)$} \\
            &= z_1^2 + \dfrac{x_1^3}{(1 + x_1)^4} - y_1^4 - \dfrac{x_1^2y_1^2}{(1 + x_1)^2}
                & &\text{$\expar{x_1 \coloneqq \dfrac{x}{1 - x}}$} \\
            &= z_1^2 + \dfrac{x_1^3}{(1 + x_1)^4} - \dfrac{y_2^4}{(1 + x_1)^4} - \dfrac{x_1^2y_2^2}{(1 + x_1)^4}
                & &\text{$(y_2 \coloneqq (1 + x_1)y_1)$} \\
            &= -\dfrac{z_2^2}{(1 + x_1)^4} + \dfrac{x_1^3}{(1 + x_1)^4} - \dfrac{y_2^4}{(1 + x_1)^4} - \dfrac{x_1^2y_2^2}{(1 + x_1)^4}
                & &\text{$(z_2 \coloneqq \sqrt{-1}(1 + x_1)^2z_1)$} \\
            &= -\dfrac{1}{(1 - x_2)^4}(z_2^2 + x_2^3 + y_2^4 + x_2^2y_2^2)
                & &\text{$(x_2 \coloneqq -x_1)$.}
        \end{align*}
        \item The Tyurina numbers in \cref{tab:Eq1,tab:Eq2} are taken from
        \cite[pp.352--353]{MR1033443} and \cite[p.356]{MR4274595} respectively.
        \item The simple singularities of dimension $2$ coincide with the rational double points.
    \end{items}
\end{rem}

\begin{thm}[{\cite[Chapter~4]{MR4390795} and {\cite[Theorem~6.1]{MR570778}}}]
    The singularity category $\DD_{sg}(S_d / \exgen{f})$ is triangulated equivalent
    to the stable category $\sMCM(S_d / \exgen{f})$ of maximal Cohen--Macaulay $S_d / \exgen{f}$-modules
    and to the homotopy category $\HMF_{S_d} (f)$ of matrix factorizations of the polynomial $f$.
\end{thm}

\begin{rem}
    These categories are $\Hom$-finite $k$-linear Krull--Schmidt categories.
\end{rem}

\begin{thm}[{\cite{MR0977477}, cf. \cite{MR877010}}]
    \label{thm:Periodicity}
    As $k$-linear triangulated categories,
    \[ \DD_{sg} (S_d / \exgen{f}) \simeq \DD_{sg} (S_{d + 2} / \exgen{f + x_{d + 1} x_{d + 2}}). \]
\end{thm}

\begin{thm}[{\cite{MR825715}, \cite{MR818299} and \cite{MR0887498}}]
    \label{thm:AR}
    Assume that $S_d / \exgen{f}$ is a simple singularity of dimension $1$ \textup{(}resp. $2$\textup{)}.
    Then the Auslander--Reiten quiver of the singularity category $\DD_{sg} (S_d / \exgen{f})$
    is given in \cref{tab:AR1} \textup{(}resp. \cref{tab:AR2}\textup{)}.
    Here, the Auslander--Reiten translation is depicted by dashed arrows.
    \input{tables/AR1}
    \input{tables/AR2}
\end{thm}

\begin{thm}[{\cite[Theorem~2.49 (3)]{MR3877165}}]
    \label{thm:HHPeriodicity}
    The dg singularity category $\DD_{sg}^{dg}(S_d / \exgen{f})$ is quasi-equivalent
    to the dg category $\MF_{S_d}^{\mathrm{dg}} (f)$ of matrix factorizations of the polynomial $f$.
    In particular, the Hochschild cohomology $\HH^{\bullet}(\DD_{sg}^{dg}(S_d / \exgen{f}))$ is $2$-periodic\textup{:}
    As $S_d / \exgen{f}$-modules,
    \begin{align*}
        &\HH^n(\DD_{sg}^{dg}(S_d / \exgen{f})) \cong \HH^{n + 2}(\DD_{sg}^{dg}(S_d / \exgen{f}))
            & &\text{for any $n \in \ZZ$.}
    \end{align*}
\end{thm}

\begin{thm}[{\cite{MR4404125}}]
    \label{thm:TriToHH}
    Assume that $S_d / \exgen{f}$ is a simple singularity.
    Then the singularity category $\DD_{sg}(S_d / \exgen{f})$ admits a unique dg enhancement.
    In particular for simple singularities $S_d / \exgen{f}$ and $S_d / \exgen{g}$,
    if their singularity categories $\DD_{sg}(S_d / \exgen{f})$ and $\DD_{sg}(S_d / \exgen{g})$
    are triangulated equivalent, then as graded $k$-algebras,
    \[
        \HH^{\bullet}(\DD_{sg}^{dg}(S_d / \exgen{f})) \cong \HH^{\bullet}(\DD_{sg}^{dg}(S_d / \exgen{g})).
    \]
\end{thm}

%% file: tables/Eq1.tex
\begin{longtable}{llll}
    \caption{Simple singularities of dimension $1$.}
    \label{tab:Eq1}
    \endfirsthead
    \multicolumn{4}{l}{\underline{Characteristic $0$ or not less than $7$}} \\ 
    & Type & Defining polynomial & Tyurina number \\
    & $A_n^0$ \textup{(}$n \ge 1$\textup{)} & $x^2+y^{n+1}$ & $n + \delta_{p, \gcd(p, n + 1)}$ \\
    & $D_n^0$ \textup{(}$n \ge 4$\textup{)} & $x^2y+y^{n-1}$ & $n$ \\
    & $E_6^0$ & $x^3+y^4$ & $6$ \\
    & $E_7^0$ & $x^3+xy^3$ & $7$ \\
    & $E_8^0$ & $x^3+y^5$ & $8$ \\
    & & & \\
    \multicolumn{4}{l}{\underline{Characteristic $5$}} \\ 
    & $A_n^0$ \textup{(}$n \ge 1$\textup{)} & $x^2+y^{n+1}$ & $n + \delta_{p, \gcd(p, n + 1)}$ \\
    & $D_n^0$ \textup{(}$n \ge 4$\textup{)} & $x^2y+y^{n-1}$ & $n$ \\
    & $E_6^0$ & $x^3+y^4$ & $6$ \\
    & $E_7^0$ & $x^3+xy^3$ & $7$ \\
    & $E_8^0$ & $x^3+y^5$ & $10$ \\
    & $E_8^1$ & $x^3+y^5+xy^4$ & $8$ \\
    & & & \\
    \multicolumn{4}{l}{\underline{Characteristic $3$}} \\ 
    & $A_n^0$ \textup{(}$n \ge 1$\textup{)} & $x^2+y^{n+1}$ & $n + \delta_{p, \gcd(p, n + 1)}$ \\
    & $D_n^0$ \textup{(}$n \ge 4$\textup{)} & $x^2y+y^{n-1}$ & $n$ \\
    & $E_6^0$ & $x^3+y^4$ & $9$ \\
    & $E_6^1$ & $x^3+y^4+x^2y^2$ & $7$ \\
    & $E_7^0$ & $x^3+xy^3$ & $9$ \\
    & $E_7^1$ & $x^3+xy^3+x^2y^2$ & $7$ \\
    & $E_8^0$ & $x^3+y^5$ & $12$ \\
    & $E_8^1$ & $x^3+y^5+x^2y^3$ & $10$ \\
    & $E_8^2$ & $x^3+y^5+x^2y^2$ & $8$ \\
    & & & \\
    \multicolumn{4}{l}{\underline{Characteristic $2$}} \\ 
    & $A_{2n-1}^0$ \textup{(}$n \ge 1$\textup{)} & $x^2+xy^n$ & $2n-1+\delta_{2, \gcd(2,n)}$ \\
    & $A_{2n}^r$ \textup{(}$n \ge 1$ and $0 \le r < n$\textup{)} & $x^2+y^{2n+1}+(1-\delta_{0, r})xy^{2n-r}$ & $4n-2r-1+\delta_{2, \gcd(2,r)}$ \\
    & $D_{2n}^0$ \textup{(}$n \ge 2$\textup{)} & $x^2y+xy^n$ & $2n$\\
    & $D_{2n+1}^r$ \textup{(}$n \ge 2$ and $0 \le r < n$\textup{)} \quad & $x^2y+y^{2n}+(1-\delta_{0, r})xy^{2n-r} \quad$ & $4n-2r$ \\
    & $E^0_6$ & $x^3+y^4$ & $8$ \\
    & $E^1_6$ & $x^3+y^4+xy^3$ & $6$ \\
    & $E_7^0$ & $x^3+xy^3$ & $7$ \\
    & $E_8^0$ & $x^3+y^5$ & $8$
\end{longtable}

%% file: tables/Eq2.tex
\begin{longtable}{llll}
    \caption{Simple singularities of dimension $2$.}
    \label{tab:Eq2}
    \endfirsthead
    \multicolumn{4}{l}{\underline{Characteristic $0$ or not less than $7$}} \\ 
    & Type & Defining polynomial & Tyurina number \\
    & $A_n^0$ \textup{(}$n \ge 1$\textup{)} & $x^{n+1}+yz$ & $n + \delta_{p, \gcd(p, n + 1)}$ \\
    & $D_{2n}^0$ \textup{(}$n \ge 2$\textup{)} & $z^2+x^2y+xy^n$ & $2n$ \\
    & $D_{2n+1}^0$ \textup{(}$n \ge 2$\textup{)} & $z^2+x^2y+y^nz$ & $2n+1$ \\
    & $E_6^0$ & $z^2+x^3+y^2z$ & $6$ \\
    & $E_7^0$ & $z^2+x^3+xy^3$ & $7$ \\
    & $E_8^0$ & $z^2+x^3+y^5$ & $8$ \\
    & & & \\
    \multicolumn{4}{l}{\underline{Characteristic $5$}} \\ 
    & $A_n^0$ \textup{(}$n \ge 1$\textup{)} & $x^{n+1}+yz$ & $n + \delta_{5, \gcd(5, n + 1)}$ \\
    & $D_{2n}^0$ \textup{(}$n \ge 2$\textup{)} & $z^2+x^2y+xy^n$ & $2n$ \\
    & $D_{2n+1}^0$ \textup{(}$n \ge 2$\textup{)} & $z^2+x^2y+y^nz$ & $2n+1$ \\
    & $E_6^0$ & $z^2+x^3+y^2z$ & $6$ \\
    & $E_7^0$ & $z^2+x^3+xy^3$ & $7$ \\
    & $E^0_8$ & $z^2+x^3+y^5$ & $10$ \\
    & $E^1_8$ & $z^2+x^3+y^5+xy^4$ & $8$ \\
    & & & \\
    \multicolumn{4}{l}{\underline{Characteristic $3$}} \\ 
    & $A_n^0$ \textup{(}$n \ge 1$\textup{)} & $x^{n+1}+yz$ & $n + \delta_{3, \gcd(3, n + 1)}$ \\
    & $D_{2n}^0$ \textup{(}$n \ge 2$\textup{)} & $z^2+x^2y+xy^n$ & $2n$ \\
    & $D_{2n+1}^0$ \textup{(}$n \ge 2$\textup{)} & $z^2+x^2y+y^nz$ & $2n+1$ \\
    & $E^0_6$ & $z^2+x^3+y^2z$ & $9$ \\
    & $E^1_6$ & $z^2+x^3+y^2z+xyz$ & $7$ \\
    & $E^0_7$ & $z^2+x^3+xy^3$ & $9$ \\
    & $E^1_7$ & $z^2+x^3+xy^3+x^2y^2$ & $7$ \\
    & $E^0_8$ & $z^2+x^3+y^5$ & $12$ \\
    & $E^1_8$ & $z^2+x^3+y^5+x^2y^3$ & $10$ \\
    & $E^2_8$ & $z^2+x^3+y^5+x^2y^2$ & $8$ \\
    & & & \\
    \multicolumn{4}{l}{\underline{Characteristic $2$}} \\ 
    & $A_n^0$ \textup{(}$n \ge 1$\textup{)} & $x^{n+1}+yz$ & $n + \delta_{2, \gcd(2, n + 1)}$ \\
    & $D_{2n}^r$ \textup{(}$n \ge 2$ and $0 \le r < n$\textup{)} & $z^2+x^2y+xy^n+(1-\delta_{0, r})xy^{n-r}z$ & $4n-r$\\
    & $D_{2n+1}^r$ \textup{(}$n \ge 2$ and $0 \le r < n$\textup{)} \quad & $z^2+x^2y+y^nz+(1-\delta_{0, r})xy^{n-r}z \quad$ & $4n-r$ \\
    & $E^0_6$ & $z^2+x^3+y^2z$ & $8$ \\
    & $E^1_6$ & $z^2+x^3+y^2z+xyz$ & $6$ \\
    & $E^0_7$ & $z^2+x^3+xy^3$ & $14$ \\
    & $E^1_7$ & $z^2+x^3+xy^3+x^2yz$ & $12$ \\
    & $E^2_7$ & $z^2+x^3+xy^3+y^3z$ & $10$ \\
    & $E^3_7$ & $z^2+x^3+xy^3+xyz$ & $8$ \\
    & $E^0_8$ & $z^2+x^3+y^5$ & $16$ \\
    & $E^1_8$ & $z^2+x^3+y^5+xy^3z$ & $14$ \\
    & $E^2_8$ & $z^2+x^3+y^5+xy^2z$ & $12$ \\
    & $E^3_8$ & $z^2+x^3+y^5+y^3z$ & $10$ \\
    & $E^4_8$ & $z^2+x^3+y^5+xyz$ & $8$
\end{longtable}

%% file: tables/AR1.tex
\begin{longtable}{lcl}
    \caption{Auslander--Reiten quivers of $1$-dimensional simple singularities.}
    \label{tab:AR1}
    \endfirsthead
    $A_{2n-1}^0$ & $\input{diagrams/1AOdd}$ & \textup{(}$n+1$ vertices\textup{)} \\
    $A_{2n}^r$ & $\input{diagrams/1AEven}$ & \textup{(}$n$ vertices\textup{)} \\
    $D_{2n}^0$ & $\input{diagrams/1DEven}$ & \textup{(}$4n$ vertices\textup{)} \\
    $D_{2n+1}^r$ & $\input{diagrams/1DOdd}$ & \textup{(}$4n-1$ vertices\textup{)} \\
    $E_6^r$ & $\input{diagrams/1E6}$ &  \\
    $E_7^r$ & $\input{diagrams/1E7}$ &  \\
    $E_8^r$ & $\input{diagrams/1E8}$ & 
\end{longtable}

%% file: diagrams/1AOdd.tex
\begin{tikzcd}[ampersand replacement=\&, row sep=tiny]
    \&
        \&
            \&
                \&
                    \bullet \ar[shift left=0.4ex]{dl} \\
    \bullet \ar[shift left=0.4ex]{r} \ar[out=305,in=235,loop, dashed] \&
        \bullet \ar[shift left=0.4ex]{r} \ar[shift left=0.4ex]{l} \ar[out=305,in=235,loop, dashed] \&
            \cdots \ar[shift left=0.4ex]{r} \ar[shift left=0.4ex]{l} \&
                \bullet \ar[shift left=0.4ex]{l} \ar[shift left=0.4ex]{ur} \ar[shift left=0.4ex]{dr} \ar[out=305,in=235,loop, dashed] \&
                    \\
    \&
        \&
            \&
                \&
                    \bullet \ar[shift left=0.4ex]{ul} \ar[dashed, <->]{uu}
\end{tikzcd}

%% file: diagrams/1AEven.tex
\begin{tikzcd}[ampersand replacement=\&]
    \bullet \ar[shift left=0.4ex]{r} \ar[out=305,in=235,loop, dashed] \&
        \bullet \ar[shift left=0.4ex]{r} \ar[shift left=0.4ex]{l} \ar[out=305,in=235,loop, dashed] \&
            \cdots \ar[shift left=0.4ex]{r} \ar[shift left=0.4ex]{l} \&
                \bullet \ar[shift left=0.4ex]{l} \ar[out=305,in=235,loop, dashed] \ar[out=35,in=325,loop]
\end{tikzcd}

%% file: diagrams/1DEven.tex
\begin{tikzcd}[ampersand replacement=\&, row sep=tiny]
    \&
        \&
            \&
                \&
                    \&
                        \bullet \ar{dl}  \ar[dashed, <->]{dd} \\
    \bullet \ar[dashed, <->]{ddd} \ar{dddr} \&
        \bullet \ar{l} \ar{r} \ar[dashed, <->]{ddd} \&
            \bullet \ar{r} \ar{dddl} \ar[dashed, <->]{ddd} \&
                \cdots \ar{r} \ar{dddl} \&
                    \bullet \ar{dr} \ar{dddl} \ar[dashed, <->]{ddd} \ar{ddddr} \&
                        \\
    \&
        \&
            \&
                \&
                    \&
                        \bullet \ar{ddl} \\
    \&
        \&
            \&
                \&
                    \&
                        \bullet \ar[dashed, <->]{dd} \ar{uul} \\
    \bullet \ar{uuur} \&
        \bullet \ar{r} \ar{l} \&
            \bullet \ar{r} \ar{uuul} \&
                \cdots \ar{r} \ar{uuul} \&
                    \bullet \ar{uuul} \ar{ur} \ar{uuuur} \&
                        \\
    \&
        \&
            \&
                \&
                    \&
                        \bullet \ar{ul} \\
\end{tikzcd}

%% file: diagrams/1DOdd.tex
\begin{tikzcd}[ampersand replacement=\&, row sep=tiny]
    \bullet \ar[dashed, <->]{dd} \ar{ddr} \&
        \bullet \ar{r} \ar[dashed, <->]{dd} \ar{l} \&
            \bullet \ar{r} \ar{ddl} \ar[dashed, <->]{dd} \&
                \cdots \ar{r} \ar{ddl} \&
                    \bullet \ar[shift left=0.4ex]{dr} \ar{ddl} \ar[dashed, <->]{dd} \&
                        \\
    \&
        \&
            \&
                \&
                    \&
                        \bullet \ar[shift left=0.4ex]{ul} \ar[shift left=0.4ex]{dl} \ar[out=305,in=235,loop, dashed] \\
    \bullet \ar{uur} \&
        \bullet \ar{r} \ar{l} \&
            \bullet \ar{r} \ar{uul} \&
                \cdots \ar{r} \ar{uul} \&
                    \bullet \ar[shift left=0.4ex]{ur} \ar{uul} \&
                    
\end{tikzcd}

%% file: diagrams/1E6.tex
\begin{tikzcd}[ampersand replacement=\&, row sep=tiny]
    \&
        \&
            \bullet \ar{r} \ar[dashed, <->]{dd} \ar[shift left=0.4ex]{dl}\&
                \bullet \ar{ddl} \ar[dashed, <->]{dd} \\
    \bullet \ar[shift left=0.4ex]{r} \ar[out=305,in=235,loop, dashed] \&
        \bullet \ar[shift left=0.4ex]{l} \ar[shift left=0.4ex]{ur} \ar[shift left=0.4ex]{dr} \ar[out=305,in=235,loop, dashed] \&
            \&
                \\
    \&
        \&
            \bullet \ar{r} \ar[shift left=0.4ex]{ul}\&
                \bullet \ar{uul}
\end{tikzcd}

%% file: diagrams/1E7.tex
\begin{tikzcd}[ampersand replacement=\&, row sep=tiny]
    \bullet \ar{r} \ar[dashed, <->]{dd}  \&
        \bullet \ar{r} \ar{ddl} \ar[dashed, <->]{dd} \&
            \bullet \ar{r} \ar{ddl} \ar[dashed, <->]{dd} \&
                \bullet \ar{r} \ar{ddl} \ar[dashed, <->]{dd} \ar{dddl} \&
                    \bullet \ar{r} \ar{ddl} \ar[dashed, <->]{dd} \&
                        \bullet \ar{ddl} \ar[dashed, <->]{dd} \\
    \phantom{\bullet} \&
        \&
            \&
                \&
                    \&
                        \\
    \bullet \ar{r} \&
        \bullet \ar{r} \ar{uul} \&
            \bullet \ar{r} \ar{uul} \&
                \bullet \ar{r} \ar{uul} \ar{dr} \&
                    \bullet \ar{r} \ar{uul} \&
                        \bullet \ar{uul}\\
    \&
        \&
            \bullet \ar[dashed, <->]{rr} \ar{ur} \&
                \&
                    \bullet \ar{uuul} \&
                        
\end{tikzcd}

%% file: diagrams/1E8.tex
\begin{tikzcd}[ampersand replacement=\&, row sep=tiny]
    \bullet \ar{r} \ar[dashed, <->]{dd} \&
        \bullet \ar{r} \ar{ddl} \ar[dashed, <->]{dd} \&
            \bullet \ar{r} \ar{ddl} \ar[dashed, <->]{dd} \ar{dddl} \&
                \bullet \ar{r} \ar{ddl} \ar[dashed, <->]{dd} \&
                    \bullet \ar{r} \ar{ddl} \ar[dashed, <->]{dd} \&
                        \bullet \ar{r} \ar{ddl} \ar[dashed, <->]{dd} \&
                            \bullet \ar{ddl} \ar[dashed, <->]{dd} \\
    \phantom{\bullet} \&
        \&
            \&
                \&
                    \&
                        \&
                            \\
    \bullet \ar{r} \&
        \bullet \ar{r} \ar{uul} \&
            \bullet \ar{r} \ar{uul} \ar{dr} \&
                \bullet \ar{r} \ar{uul} \&
                    \bullet \ar{r} \ar{uul} \&
                        \bullet \ar{r} \ar{uul} \&
                            \bullet \ar{uul}\\
    \&
        \bullet \ar[dashed, <->]{rr} \ar{ur} \&
            \&
                \bullet \ar{uuul} \&
                    \&
                        \&
                        
\end{tikzcd}

%% file: tables/AR2.tex
\begin{longtable}{lcl}
    \caption{Auslander--Reiten quivers of $2$-dimensional simple singularities.}
    \label{tab:AR2}
    \endfirsthead
    $A_n^0$ & $\input{diagrams/2A}$ & \textup{(}$n$ vertices\textup{)} \\
    $D_n^r$ & $\input{diagrams/2D}$ & \textup{(}$n$ vertices\textup{)} \\
    $E_n^r$ & $\input{diagrams/2E}$ & \textup{(}$n$ vertices\textup{)}
\end{longtable}

%% file: diagrams/2A.tex
\begin{tikzcd}[ampersand replacement=\&]
    \bullet \ar[shift left=0.4ex]{r} \ar[out=305,in=235,loop, dashed] \&
        \bullet \ar[shift left=0.4ex]{r} \ar[shift left=0.4ex]{l} \ar[out=305,in=235,loop, dashed] \&
            \cdots \ar[shift left=0.4ex]{r} \ar[shift left=0.4ex]{l} \&
                \bullet \ar[shift left=0.4ex]{l} \ar[out=305,in=235,loop, dashed]
\end{tikzcd}

%% file: diagrams/2D.tex
\begin{tikzcd}[ampersand replacement=\&, row sep=tiny]
    \&
        \&
            \&
                \&
                    \bullet \ar[shift left=0.4ex]{dl} \ar[out=305,in=235,loop, dashed] \\
    \bullet \ar[shift left=0.4ex]{r} \ar[out=305,in=235,loop, dashed] \&
        \bullet \ar[shift left=0.4ex]{r} \ar[shift left=0.4ex]{l} \ar[out=305,in=235,loop, dashed] \&
            \cdots \ar[shift left=0.4ex]{r} \ar[shift left=0.4ex]{l} \&
                \bullet \ar[shift left=0.4ex]{l} \ar[shift left=0.4ex]{ur} \ar[shift left=0.4ex]{dr} \ar[out=305,in=235,loop, dashed] \&
                    \\
    \&
        \&
            \&
                \&
                    \bullet \ar[shift left=0.4ex]{ul} \ar[out=305,in=235,loop, dashed]
\end{tikzcd}

%% file: diagrams/2E.tex
\begin{tikzcd}[ampersand replacement=\&]
    \&
        \&
            \bullet \ar[shift left=0.4ex]{d} \ar[out=35,in=325,loop, dashed] \&
                \&
                    \&
                        \\
    \bullet \ar[shift left=0.4ex]{r} \ar[out=305,in=235,loop, dashed] \&
        \bullet \ar[shift left=0.4ex]{r} \ar[shift left=0.4ex]{l} \ar[out=305,in=235,loop, dashed] \&
            \bullet \ar[shift left=0.4ex]{r} \ar[shift left=0.4ex]{l} \ar[shift left=0.4ex]{u} \ar[out=305,in=235,loop, dashed] \&
                \bullet \ar[shift left=0.4ex]{r} \ar[shift left=0.4ex]{l} \ar[out=305,in=235,loop, dashed] \&
                    \cdots \ar[shift left=0.4ex]{r} \ar[shift left=0.4ex]{l} \&
                        \bullet \ar[shift left=0.4ex]{l} \ar[out=305,in=235,loop, dashed]
\end{tikzcd}

%% file: materials/MainResult.tex
\section{Singularity Categories of Simple Singularities}
Let $S \coloneqq k[x_0, x_1, \ldots, x_d]$ be a polynomial ring with $d \geq 0$,
$m_S \coloneqq \exgen{x_0, x_1, \ldots, x_d}_S$ a maximal ideal of $S$ and $f \in S$ a polynomial
such that $\Spec (S / \exgen{f})$ has only one singular point at the origin.
By convention, the completions of these rings are always taken with respect to the origin.
In this section, we show that the $0$th Hochschild cohomology $\HH^0(\DD_{sg}^{dg}(\widehat{S} / \exgen{f}))$
is isomorphic to the Tyurina algebra $T_f$ of the polynomial $f$ as $\widehat{S} / \exgen{f}$-algebras in arbitrary characteristic.
This allows us to distinguish the singularity categories of $d$-dimensional simple singularities
which are not analytically isomorphic.

\begin{lem}
    \begin{items}
        \label{lem:Koszul}
        \item \label{lem:Koszul:1} As $S / \exgen{f}$-modules,
        \begin{align*}
            &H_n\expar{K_{\bullet}\expar{S / \exgen{f} ; \frac{\partial f}{\partial x_0},
            \frac{\partial f}{\partial x_1}, \ldots, \frac{\partial f}{\partial x_d}}} \cong
            H_n\expar{K_{\bullet}\expar{S ;f, \frac{\partial f}{\partial x_0},
            \frac{\partial f}{\partial x_1}, \ldots, \frac{\partial f}{\partial x_d}}}
                & &\text{for $n \ge 0$.}
        \end{align*}
        \item \label{lem:Koszul:2} \begin{align*}
            &H_n\expar{K_{\bullet}\expar{S ;f, \frac{\partial f}{\partial x_0},
            \frac{\partial f}{\partial x_1}, \ldots, \frac{\partial f}{\partial x_d}}} = 0
                & &\text{for $n \ge 2$.}
        \end{align*}
    \end{items}
\end{lem}

\begin{proof}
    See \cite[Proposition~1.6.13~(b)]{MR1251956} for \cref*{lem:Koszul:1}.
    We only consider \cref*{lem:Koszul:2}. Set
    \begin{align*}
        &\mu \coloneqq \max \inset{i \ge 0}{H_i\expar{K_{\bullet}\expar{S ;f, \frac{\partial f}{\partial x_0}, \frac{\partial f}{\partial x_1}, \ldots, \frac{\partial f}{\partial x_d}}} \neq 0}, \\
        &I \coloneqq \exgen{f, \frac{\partial f}{\partial x_0}, \frac{\partial f}{\partial x_1}, \ldots, \frac{\partial f}{\partial x_d}}_{S}.
    \end{align*}
    Then
    \begin{align*}
        \mu &= d + 2 - \depth_I (S)
            & &\text{by \cite[Proposition~1.6.17 (b)]{MR1251956}} \\
        &= d + 2 - \height_S (I)
            & &\text{by \cite[Corollary~2.1.4]{MR1251956}.}
    \end{align*}
    Since $\Spec (S / \exgen{f})$ has only one singular point at the origin, it follows that
    $\VV[S](I) = \exset{m_S}$. This implies $\sqrt{I} = m_S$, and hence
    \[ \height_S (I) = \height_S (m_S) = d + 1. \]
    Combining these equalities, we get $\mu = 1$.
\end{proof}

\begin{dfn}
    Let $A$ be a finitely generated commutative $k$-algebra and $n$ an integer.
    \begin{items}
        \item The $n$-th \emph{Hochschild cohomology} of $A$ is defined by
        \[ \HH^n (A) \coloneqq \Hom_{\DD^{b}(A \tens[k] A)}(A, A[n]). \]
        \item The $n$-th \emph{singular Hochschild cohomology} of $A$ is defined by
        \[ \HH[sg]^n (A) \coloneqq \Hom_{\DD_{sg}(A \tens[k] A)}(A, A[n]). \]
    \end{items}
\end{dfn}

\begin{prop}
    \label{prop:HHToTyu}
    As $\widehat{S} / \exgen{f}$-algebras,
    \[
        \HH^0(\DD_{sg}^{dg}(\widehat{S} / \exgen{f})) \cong T_f.
    \]
\end{prop}

\begin{proof}
    We adopt an approach similar to that of \cite[Theorem~5.9]{MR4817470} (see \cref{rem:HK}).
    By \cite[Corollary~5.4]{MR4817470}, there exists an integer $n \ge 1$ such that as $S_{m_S} / \exgen{f}$-modules
    \begin{align*}
        &\HH[sg]^i (S_{m_S} / \exgen{f}) \cong \HH^i (S_{m_S} / \exgen{f})
            & &\text{for any $i \ge n$.}
    \end{align*}
    Also by \cite[Theorem~3.2.7 (2)]{MR1176152} and \cref{lem:Koszul},
    we get the following $S / \exgen{f}$-isomorphisms for any integer $i$ with $2i \ge d + 1$:
    \begin{align*}
        \HH^{2i} (S / \exgen{f})
        &\cong \bigoplus_{i \leq j \leq (2i + d + 1) / 2} H_{2j - 2i} \expar{K_{\bullet}
        \expar{S / \exgen{f}; \frac{\partial f}{\partial x_0}, \frac{\partial f}{\partial x_1}, \ldots, \frac{\partial f}{\partial x_d}}} \\
        &\cong H_0 \expar{K_{\bullet}\expar{S; f, \frac{\partial f}{\partial x_0}, \frac{\partial f}{\partial x_1}, \ldots, \frac{\partial f}{\partial x_d}}} \\
        &\cong S \bigg/ \exgen{f, \frac{\partial f}{\partial x_0}, \frac{\partial f}{\partial x_1}, \ldots, \frac{\partial f}{\partial x_d}}.
    \end{align*}
    Take an integer $i$ satisfying $2i \ge \max \exset{n, d + 1}$.
    Note that since the canonical functor
    $\overline{\DD_{sg}(S_{m_S} / \exgen{f})} \to \DD_{sg}(\widehat{S} / \exgen{f})$
    is a triangulated equivalence
    (see \cite[Proposition~2.7]{MR2735755}, \cite[Proposition~A.1]{MR2776613} and \cite[Lemma~5.6]{MR2824483}),
    it follows that the dg singularity categories $\DD_{sg}^{dg}(S_{m_S} / \exgen{f})$ and
    $\DD_{sg}^{dg}(\widehat{S} / \exgen{f})$ are Morita equivalent (e.g., \cite[Proposition~4.4.6]{MR2762557}).
    Then as $S_{m_S} / \exgen{f}$-modules,
    \begin{align*}
        &\HH^0 (\DD_{sg}^{dg}(\widehat{S} / \exgen{f})) \\
        &\cong \HH^{2i} (\DD_{sg}^{dg}(\widehat{S} / \exgen{f}))
            & &\text{by \cref{thm:HHPeriodicity}} \\
        &\cong \HH^{2i} (\DD_{sg}^{dg}(S_{m_S} / \exgen{f})) \\
        &\cong \HH[sg]^{2i} (S_{m_S} / \exgen{f})
            & &\text{by \cite[Theorem 1]{MR3907577}} \\
        &\cong \HH^{2i} (S_{m_S} / \exgen{f})
            & &\text{by $2i \ge n$} \\
        &\cong \HH^{2i} (S / \exgen{f})_{m_S}
            & &\text{by \cite[Lemma 5.6]{MR4817470}} \\
        &\cong S_{m_S} \bigg/ \exgen{f, \frac{\partial f}{\partial x_0}, \frac{\partial f}{\partial x_1}, \ldots, \frac{\partial f}{\partial x_d}}
            & &\text{by $2i \ge d + 1$} \\
        &\cong T_f.
    \end{align*}
    As a consequence, the induced $S_{m_S} / \exgen{f}$-linear isomorphism $\HH^0 (\DD_{sg}^{dg}(\widehat{S} / \exgen{f})) \cong T_f$
    is also an isomorphism as $\widehat{S} / \exgen{f}$-algebras.
\end{proof}

\begin{rem}
    \label{rem:HK}
    In characteristic $0$, \cref{prop:HHToTyu} is shown in \cite[Theorem~5.9]{MR4817470}.
    Although the characteristic $0$ assumption is not explicitly imposed,
    their proof may fail in positive characteristic:
    they use \cite[Theorem~5.5]{MR4817470} to prove \cite[Theorem~5.9]{MR4817470} and
    it depends on the property that
    the partial derivatives $\partial f /\partial x_0, \partial f /\partial x_1, \ldots, \partial f /\partial x_d$
    constitute a regular sequence in $S$.
    This is not true in positive characteristic in general.
\end{rem}

\begin{thm}
    \label{thm:LongWork}
    For simple singularities $\widehat{S} / \exgen{f}$ and $\widehat{S} / \exgen{g}$ of dimension $d$, they are isomorphic
    if and only if their singularity categories $\DD_{sg}(\widehat{S} / \exgen{f})$ and $\DD_{sg}(\widehat{S} / \exgen{g})$
    are triangulated equivalent.
\end{thm}

\begin{proof}
    Since the \emph{only if} implication is obvious, we only prove the converse.
    It suffices to consider the case where $d = 1$ or $2$ (see \cref{eg:ZeroSimple}, \cref{thm:Classification} \cref*{thm:Classification:2}
    and \cref{thm:Periodicity}), and the Auslander--Reiten quivers of the singularity categories
    $\DD_{sg}(\widehat{S} / \exgen{f})$ and $\DD_{sg}(\widehat{S} / \exgen{g})$ are the same (see \cref{thm:AR}).
    Assume that $\DD_{sg}(\widehat{S} / \exgen{f})$ and $\DD_{sg}(\widehat{S} / \exgen{g})$ are triangulated equivalent.
    Then \cref{thm:TriToHH} and \cref{prop:HHToTyu} yield
    \[ 
        \tau(f) = \rk_k \HH^0 (\DD_{sg}^{dg}(\widehat{S} / \exgen{f}))
                = \rk_k \HH^0 (\DD_{sg}^{dg}(\widehat{S} / \exgen{g}))
                = \tau(g).
    \]
    By \cref{tab:Eq1,tab:Eq2}, we obtain
    \[ \widehat{S} / \exgen{f} \cong \widehat{S} / \exgen{g}. \qedhere \]
\end{proof}

%% file: materials/App.tex
\section{(Non-)Standardness of Singularity Categories of Rational Double Points}
As an application of \cref{thm:LongWork}, we will determine when the singularity category
of a rational double point (i.e., a simple singularity of dimension $2$) is standard.
Let $S$ denote the polynomial ring $k[x, y, z]$,
$R$ the quotient ring $S / \exgen{f}$ with $f$ one of the polynomials listed in \cref{tab:Eq2},
$T_n^r$ ($T \in \exset{A, D, E}$, $n \ge 1$ and $r \ge 0$) the type of the rational double point $S / \exgen{f}$,
and $Q$ a quiver whose underlying graph is the Dynkin graph of type $T$.
By convention, the completions of the rings $S$ and $R$ are always taken with respect to the origin.
Note that the polynomial $f$ is weighted homogeneous if and only if $r = 0$.
In this case, $\MCM^{\ZZ} (R)$ (resp. $\sMCM^{\ZZ} (R)$) stands for the (resp. stable) category of
graded maximal Cohen--Macaulay $R$-modules, and the $\Hom$-set from one object $M$ to another $N$
in $\MCM^{\ZZ} (R)$ (resp. $\sMCM^{\ZZ} (R)$) is denoted by $\Hom_R^{\ZZ} (M, N)$ (resp. $\sHom_R^{\ZZ} (M, N)$).

\begin{dfn}
    Let $\A$ be a Krull--Schmidt category.
    The full subcategory of $\A$ consisting of a chosen representative
    from each isomorphism class of indecomposable objects is denoted by $\ind (\A)$.
    In the case where $\# \ind (\A) < \infty$, the \emph{Auslander algebra} $\Aus(\A)$ of $\A$ is defined to be
    the endomorphism ring $\End_\A (\bigoplus_{M \in \ind(\A)} M)$.
\end{dfn}

\begin{dfn}
    \label{dfn:Std}
    A $\Hom$-finite $k$-linear Krull--Schmidt category $\A$
    is \emph{standard} if
    \[ \ind(\A) \simeq k(\varGamma, \tau), \]
    where $(\varGamma, \tau)$ is the Auslander--Reien quiver of $\A$.
\end{dfn}

\begin{rem}
    \label{rem:Std}
    In the case where $\# \ind (\A) < \infty$,
    the $\Hom$-finite $k$-linear Krull--Schmidt category $\A$ is standard
    if and only if $\Aus(\A) \cong k[\varGamma, \tau]$ (e.g., \cite[Proposition~2.3.28]{MR4486377}).
\end{rem}

\begin{prop}[{\cite[Theorem~3.3]{MR4172701}, cf. \cite[Corollary~2]{MR3864195}}]
    \label{prop:Recover}
    Set
    \begin{align*}
        &\mathsf{Cat} \coloneqq \inset{\A}{
            \begin{array}{ll}
                \textup{$\A$ is a standard $\Hom$-finite $k$-linear Krull--Schmidt}\\
                \textup{algebraic triangulated category satisfying $\# \ind (\A) < \infty$}
            \end{array}
        }, \\
        &\mathsf{Alg} \coloneqq \inset{k[\varGamma, \tau]}{
            \textup{$(\varGamma, \tau)$ is a finite translation quiver
            satisfying $\rk_k k[\varGamma, \tau] < \infty$}
        }.
    \end{align*}
    Then the mapping
    \[
        \begin{array}{ccc}
            \mathsf{Cat} / \textup{(triangulated equivalence)} & \rightleftarrows & \mathsf{Alg} / \mathord{\cong} \\
            \A & \mapsto & \Aus(\A) \\
            \proj k[\varGamma, \tau] & \mapsfrom & k[\varGamma, \tau]\\
        \end{array}
    \]
    is a well-defined bijection.
\end{prop}

\begin{dfn}
    Let $\A$ be a $k$-linear category and $\tau \colon \A \to \A$ an equivalence.
    The \emph{orbit category} $\A / \tau$ of $\A$ under $\tau$ is defined as follows:
    \begin{itemize}
        \item the objects in $\A / \tau$ consist of those in $\A$;
        \item $\Hom_{\A / \tau} (M, N) \coloneqq \bigoplus_{n \in \ZZ} \Hom_{\A} (M, \tau^n(N))$
        for objects $M$ and $N$ in $\A / \tau$.
    \end{itemize}
    For a morphism $a \colon M \to N$ in $\A / \tau$, its homogeneous component of degree $n \in \ZZ$ is denoted by $a_n$:
    \begin{align*}
        &a = \sum_{n \in \ZZ} a_n
            & &\text{where $a_n \in \Hom_{\A} (M, \tau^n(N))$.}
    \end{align*}
\end{dfn}

\begin{rem}
    \label{rem:OrbitUniv}
    The canonical projection $\pi \colon \A \to \A / \tau$ is endowed with a natural isomorphism
    $\pi \circ \tau \cong \pi$ and is $2$-universal among such functors:
    for any functor $\phi \colon \A \to \B$ equipped with a natural isomorphism $\phi \circ \tau \cong \phi$,
    there exists a functor $\overline{\phi} \colon \A / \tau \to \B$,
    unique up to natural isomorphism, such that $\overline{\phi} \circ \pi \cong \phi$.
\end{rem}

\begin{lem}
    \label{lem:OrbitBasic}
    Let $\A$ be a $\Hom$-finite $k$-linear Krull--Schmidt category and $\tau \colon \A \to \A$ an equivalence
    such that the orbit category $\A / \tau$ is a $\Hom$-finite $k$-linear category.
    \begin{items}
        \item \label{lem:OrbitBasic:1} Let $M$ be an object of $\A$.
        Then $M$ is indecomposable in $\A / \tau$ if and only if so is in $\A$.
        \item \label{lem:OrbitBasic:2} Let $M$ and $N$ be indecomposable objects in $\A$.
        Then $M \cong N$ in $\A / \tau$ if and only if $M \cong \tau^n(N)$ for some $n \in \ZZ$.
    \end{items}
\end{lem}

\begin{proof}
    The \emph{only if} implication of \cref*{lem:OrbitBasic:1} and the \emph{if} implication of \cref*{lem:OrbitBasic:2} are obvious.

    \emph{Step 1.} We first consider the \emph{if} implication of \cref*{lem:OrbitBasic:1}.
    Assume that the object $M$ is indecomposable in $\A$, and let
    \begin{align*}
        &\begin{tikzcd}[ampersand replacement=\&, column sep = huge]
            M \ar[shift left=0.4ex]{r}{        
                \begin{pmatrix}
                    a \\
                    a'
                \end{pmatrix}
            } \&
                N \oplus N' \ar[shift left=0.4ex]{l}{
                    \begin{pmatrix}
                        b & b'
                    \end{pmatrix}
                }
        \end{tikzcd}
            & \text{in $\A / \tau$.}
    \end{align*}
    be mutually inverse isomorphisms. Then
    \begin{align*}
        &(b \circ a)_0 + (b' \circ a')_0 = \id_M
            & &\text{in $\A$.}
    \end{align*}
    Since $M$ is indecomposable in the Krull--Schmidt category $\A$,
    it follows that the endomorphism ring $\End_{\A} (M)$ is local,
    and hence either $(b \circ a)_0$ or $(b' \circ a')_0$ is an isomorphism.
    Without loss of generality, we can assume that $(b \circ a)_0 = \sum_{m \in \ZZ} \tau^m (b_{-m}) \circ a_m$
    is an isomorphism.
    This yields a splitting monomorphism $M \to \oplus_{m \in \ZZ} \tau^m (N)$ in $\A$.
    Then $M$ is a direct summand of $\tau^m(N)$ in $\A$ for some $m \in \ZZ$.
    This implies that there exists an object $M'$ such that $N \cong M \oplus M'$ in $\A / \tau$.
    As a consequence,
    \begin{align*}
        &M \cong M \oplus M' \oplus N'
            & &\text{in $\A / \tau$.}
    \end{align*}
    Considering that $\A / \tau$ is a $\Hom$-finite $k$-linear category,
    we obtain $ M' \oplus N' = 0$. In particular, $N' = 0$.

    \emph{Step 2.} Next, we consider the \emph{only if} implication of \cref*{lem:OrbitBasic:2}.
    Assume $M \cong N$ in $\A / \tau$, and let $a \colon M \to N$ and $b \colon N \to M$
    be mutually inverse isomorphisms in $\A / \tau$. Then
    \begin{align*}
        &\id_M = b \circ a = \sum_{m, n \in \ZZ} \tau^m (b_n) \circ a_m
            & &\text{in $\A / \tau$}.
    \end{align*}
    In particular,
    \begin{align*}
        &\id_M = \sum_{m \in \ZZ} \tau^m (b_{-m}) \circ a_m
            & &\text{in $\A$}.
    \end{align*}
    Since the indecomposability of $M$ and the Krull--Schmidt property of $\A$
    imply that the endomorphism ring $\End_{\A} (M)$ is local,
    it follows that $\tau^m (b_{-m}) \circ a_m \colon M \to M$ is an isomorphism for some $m \in \ZZ$.
    Therefore, $a_m \colon M \to \tau^m (N)$ is a splitting monomorphism.
    Considering that $\tau^m (N)$ is indecomposable, we obtain that $a_m \colon M \to \tau^m (N)$ is an isomorphism.
\end{proof}

\begin{prop}[{\cite{MR2890586}, cf. \cite{MR2313537}}]
    \label{prop:StartingPoint}
    Assume $r = 0$. Then as $k$-linear triangulated categories,
    \[ \sMCM^{\ZZ} (R) \simeq \DD^{b} (\mod k[Q]). \]
    Moreover, the Auslander--Reiten translation of $\sMCM^{\ZZ} (R)$ is given by the negative degree shift \linebreak
    $(-1) \colon \sMCM^{\ZZ} (R) \to \sMCM^{\ZZ} (R)$.
\end{prop}

\begin{prop}[{Cf. \cite[Proposition~1.5]{MR2776613}}]
    \label{prop:MCMOrbit}
    Assume $r = 0$. Then
    \[ \sMCM^{\ZZ} (R) / (-1) \simeq \sMCM(\widehat{R}). \]
\end{prop}

\begin{proof}
    Let
    \[
        \phi \colon
        \begin{array}{ccc}
            \sMCM^{\ZZ} (R)
                & \to
                    & \sMCM(\widehat{R}) \\
            M = \displaystyle \bigoplus_{n \in \ZZ} M_n
                & \mapsto
                    & \widehat{M} = \displaystyle \prod_{n \in \ZZ} M_n
        \end{array}
    \]
    be the functor induced from the completion with respect to the origin of $\Spec R$.
    Then by \cref{rem:OrbitUniv}, the equality $\phi \circ (-1) = \phi$ yields a unique functor
    $\overline{\phi} \colon \sMCM^{\ZZ} (R) / (-1) \to \sMCM(\widehat{R})$ such that
    \[ 
        \begin{tikzcd}[ampersand replacement=\&, column sep=tiny]
            \sMCM^{\ZZ} (R) \ar{rr}{\phi} \ar{dr}{} \&
                \ar[phantom]{d}{\circlearrowleft} \&
                    \sMCM(\widehat{R}) \\
            \&
                {\sMCM^{\ZZ} (R) / (-1)}\rlap{.} \ar{ur}[']{\overline{\phi}}
        \end{tikzcd}
    \]
    Since the functor $\phi \colon \sMCM^{\ZZ} (R) \to \sMCM(\widehat{R})$ is essentially surjective
    by \cite[Lemma~1.3~(d) and Theorem~3.2~(a)]{MR1171231},
    so is $\overline{\phi} \colon \sMCM^{\ZZ} (R) / (-1) \to \sMCM(\widehat{R})$.
    In what follows, we prove that the functor $\overline{\phi} \colon \sMCM^{\ZZ} (R) / (-1) \to \sMCM(\widehat{R})$
    is fully faithful.
    Take objects $M$ and $N$ in $\sMCM^{\ZZ} (R) / (-1)$. Then the diagram
    \[
        \begin{tikzcd}[ampersand replacement=\&]
            \displaystyle \bigoplus_{n \in \ZZ} \sHom_{R}^{\ZZ} (M, N(-n)) \ar{r}{\overline{\phi}} \ar[hookrightarrow]{d} \&
                \sHom_{\widehat{R}} (\widehat{M}, \widehat{N}) \\
            \displaystyle \prod_{n \in \ZZ} \sHom_R^{\ZZ} (M, N(-n)) \ar[equal]{r}{} \&
                \sHom_R (M, N)^{\wedge} \ar[u, "\sim"' sloped]
        \end{tikzcd}
    \]
    is commutative, and the canonical inclusion $\bigoplus_{n \in \ZZ} \sHom_{R}^{\ZZ} (M, N(-n)) \hookrightarrow
    \prod_{n \in \ZZ} \sHom_R^{\ZZ} (M, N(-n))$ is surjective by
    \[
        \rk_k \prod_{n \in \ZZ} \sHom_R^{\ZZ} (M, N(-n))
        = \rk_k \sHom_R (M, N)^{\wedge}
        = \rk_k \sHom_{\widehat{R}} (\widehat{M}, \widehat{N})
        < \infty. \qedhere
    \]
\end{proof}

Let the Auslander--Reiten translation of $\DD^{b} (\mod k[Q])$ be denoted by $\tau \colon \DD^{b} (\mod k[Q]) \to \DD^{b} (\mod k[Q])$,
and the Auslander--Reiten quiver of $\DD^{b} (\mod k[Q])$ by $(\ZZ Q, \tau)$ (see \cite[p.54]{MR0935124}).
Here,
\begin{itemize}
    \item the set $(\ZZ Q)_0$ of vertices in the quiver $\ZZ Q$ is $\ZZ \times Q$;
    \item any arrow $M \to N$ in $Q$ yields arrows in $\ZZ Q$
    \begin{align*}
        &\text{$(n, M) \to (n, N)$ and $(n-1, M) \to (n, N)$}
            & &\text{for $n \in \ZZ$,}
    \end{align*}
    and there are no arrows in $\ZZ Q$ otherwise;
    \item the Auslander--Reiten translation $\tau \colon \DD^{b} (\mod k[Q]) \to \DD^{b} (\mod k[Q])$ acts on $(\ZZ Q)_0$ by
    \begin{align*}
        &\tau(n, M) = (n-1, M)
            & &\text{for $(n, M) \in (\ZZ Q)_0$.}
    \end{align*}
\end{itemize}

\begin{thm}
    \label{thm:RDPStd}
    The singularity category $\DD_{sg} (\widehat{R})$ is standard if and only if $r = 0$.
\end{thm}

\begin{proof}
    \emph{Step 1.} We first show the \emph{if} implication.
    Assume $r = 0$. Since the Auslander--Reiten translation of $\sMCM^{\ZZ} (R)$ (resp. $\DD^{b} (\mod k[Q])$)
    is given by the composition of the Serre functor and the negative shift $[-1]$,
    and the Serre functor is uniquely determined,
    it follows that the Auslander--Reiten translations of $\sMCM^{\ZZ} (R)$ and $\DD^{b} (\mod k[Q])$
    are canonically identified via the triangulated equivalence in \cref{prop:StartingPoint}.
    Therefore, \cref{prop:StartingPoint,prop:MCMOrbit} yield
    \[ \DD_{sg} (\widehat{R}) \simeq \sMCM(\widehat{R}) \simeq \sMCM^{\ZZ} (R) / (-1) \simeq \DD^{b} (\mod k[Q]) / \tau. \]
    Therefore, it suffices to show that the orbit category $\DD^{b} (\mod k[Q]) / \tau$ is standard.
    Since the Auslander--Reiten quiver of $\DD^{b} (\mod k[Q]) / \tau$ is $(\ZZ Q / \tau, \id)$
    by \cite[Proposition~1.3]{MR2249625},
    we need to show that $\ind (\DD^{b} (\mod k[Q]) / \tau) \simeq k(\ZZ Q / \tau, \id)$.
    Note that $\ind (\DD^{b} (\mod k[Q])) \simeq k(\ZZ Q, \tau)$ by \cite[p54]{MR0935124},
    and we identify these categories.
    Let $\pi \colon k(\ZZ Q, \tau) \to k(\ZZ Q / \tau, \id)$ be the functor
    induced by the canonical projection $\ZZ Q \to \ZZ Q / \tau$.
    Then by \cref{rem:OrbitUniv}, the equality $\pi \circ \tau = \pi$ yields
    a unique functor $\phi \colon \ind (\DD^{b} (\mod k[Q])) / \tau \to k(\ZZ Q / \tau, \id)$ such that
    \begin{align*}
        \begin{tikzcd}[ampersand replacement=\&]
            \ind (\DD^{b} (\mod k[Q])) \ar{r}{} \ar[phantom]{dr}{\circlearrowleft} \&
                \ind (\DD^{b} (\mod k[Q])) / \tau \ar{d}{\phi} \\
            k(\ZZ Q, \tau) \ar[']{r}{\pi} \ar[equal]{u}{} \&
                k(\ZZ Q / \tau, \id)\rlap{.}
        \end{tikzcd}
    \end{align*}
    The canonical projection $\ZZ Q \to \ZZ Q / \tau$ is a covering in the sense of \cite[1.3~Definition]{MR0643558},
    and hence \cite[ 3.1 Examples (a)]{MR0643558} implies that
    $\phi \colon \ind (\DD^{b} (\mod k[Q])) / \tau \to k(\ZZ Q / \tau, \id)$ is fully faithful.
    This functor being essentially surjective by the construction, we obtain that
    \[ \ind (\DD^{b} (\mod k[Q])) / \tau \simeq k(\ZZ Q / \tau, \id). \]
    On the other hand by \cref{lem:OrbitBasic},
    \[ \ind (\DD^{b} (\mod k[Q])) / \tau \simeq \ind (\DD^{b} (\mod k[Q]) / \tau). \]
    Combining these equivalences, we get the desired result.

    \emph{Step 2.} Next, we consider the \emph{only if} implication.
    Assume that the singularity category $\DD_{sg} (\widehat{S} / \exgen{f})$ is standard
    and let $\widehat{S} / \exgen{g}$ be the rational double point of type $T_n^0$.
    Since the singularity category $\DD_{sg} (\widehat{S} / \exgen{g})$ is standard by Step 1,
    it follows that \cref{rem:Std} and \cref{prop:Recover} yield the triangulated equivalence
    \[ \DD_{sg} (\widehat{S} / \exgen{f}) \simeq \DD_{sg} (\widehat{S} / \exgen{g}). \]
    Combining this with \cref{thm:LongWork}, we get $\widehat{S} / \exgen{f} \cong \widehat{S} / \exgen{g}$.
\end{proof}